\documentclass[12pt]{amsart}
\usepackage{latexsym,fancyhdr,amssymb,color,amsmath,amsthm,graphicx,listings,comment}
\newtheorem{thm}{Theorem}       \newtheorem{propo}{Proposition}
\newtheorem{lemma}{Lemma}       \newtheorem{coro}{Corollary}
\usepackage[section]{placeins}      \let\paragraph\subsection
\title{A parametrized Poincar\'e Hopf Theorem and Clique Cardinalities of graphs}
\author{Oliver Knill} \date{6/15/2019}
\address{Department of Mathematics \\ Harvard University \\ Cambridge, MA, 02138 }

\begin{document}

\begin{abstract}
Given a locally injective real function $g$ on the vertex set $V$ 
of a finite simple graph $G=(V,E)$, we prove the Poincar\'e-Hopf formula
$f_G(t) = 1+t \sum_{x \in V} f_{S_g(x)}(t)$, where 
$S_g(x) = \{ y \in S(x), g(y)<g(x) \}$ and 
$f_G(t)=1+f_0 t + \dots + f_{d} t^{d+1}$ is the $f$-function encoding 
the $f$-vector of a graph $G$, where $f_k$ counts the number of 
$k$-dimensional cliques, complete sub-graphs, in $G$. 
The corresponding computation of $f$ reduces the problem recursively 
to $n$ tasks of graphs of half the size. 
For $t=-1$, the parametric Poincar\'e-Hopf formula reduces to the classical 
Poincar\'e-Hopf result \cite{poincarehopf} $\chi(G)=\sum_x i_g(x)$, 
with integer indices $i_g(x)=1-\chi(S_g(x))$ and Euler characteristic $\chi$. 
In the new Poincar\'e-Hopf formula, the indices are integer polynomials 
and the curvatures $K_x(t)$ expressed as index expectations 
$K_x(t) = {\rm E}[i_x(t)]$ are polynomials over $\mathbb{Q}$. 
Integrating the Poincar\'e-Hopf formula over probability spaces of 
functions $g$ gives Gauss-Bonnet formulas like 
$f_G(t) = 1+\sum_{x} F_{S(x)}(t)$, where $F_G(t)$ is the anti-derivative of 
$f$ \cite{cherngaussbonnet,dehnsommervillegaussbonnet}.
A similar computation holds for the generating function 
$f_{G,H}(t,s) = \sum_{k,l} f_{k,l}(G,H) s^k t^l$
of the $f$-intersection matrix $f_{k,l}(G,H)$ counting the number 
of intersections of $k$-simplices in $G$ with $l$-simplices in $H$.
Also here, the computation is reduced to $4 n^2$ computations for 
graphs of half the size:
$f_{G,H}(t,s) = \sum_{v,w} f_{B_g(v),B_g(w)}(t,s) - f_{B_g(v),S_g(w)}(t,s) 
- f_{S_g(v),B_g(w)}(t,s) + f_{S_g(v),S_g(w)}(t,s)$, where 
$B_g(v)= S_g(v)+\{v\}$ is the unit ball of $v$. 
\end{abstract} 
\maketitle

\section{Introduction}

\paragraph{}
Given a {\bf finite simple graph} $G=(V,E)$ with {\bf vertex set} $V$ and {\bf edge set} $E$, 
the complete sub-graphs of $G$ are the simplices of a finite abstract simplicial complex $G$, 
the {\bf Whitney complex} of $G$. In graph theory, the complete sub-graphs are also known as 
{\bf cliques} of $G$ and the Whitney complex is also known as the {\bf clique complex}. 
If $f_k$ is the number of $k$-dimensional cliques in $G$, then $(f_0, \dots, f_d)$ is called the 
{\bf $f$-vector} of $G$ and the integer $d$ is the {\bf maximal dimension} of the graph $G$. 

\paragraph{}
How fast can we compute $f_G$? The problem is difficult because
a computation of $f$ also reveals the size of the maximal clique in $G$, which is 
known to be a {\bf NP hard problem} \cite{DuKo}. 
The fastest known algorithms to compute $f$ are exponential in $n$. 
We give here a formula which is sub-exponential in typical cases. 
It uses the fact that for a random function $g$, we expect half of the vertices
in $S(x)$ to be in $S_g(x)$, so that in each dimension reduction step, we have
$n$ tasks of graphs which are expected to have half the size. 

\begin{figure}[!htpb]
\scalebox{0.4}{\includegraphics{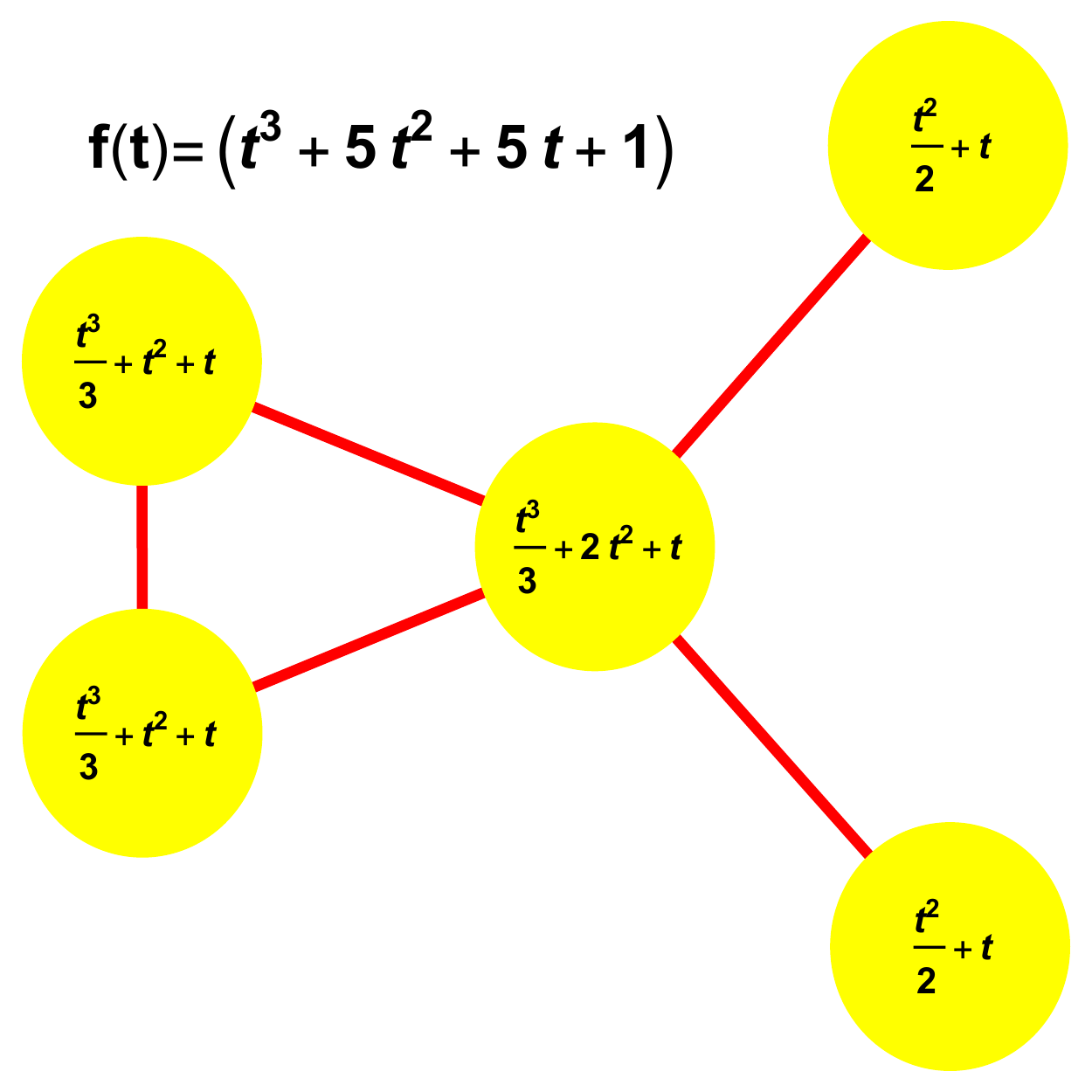}}
\scalebox{0.4}{\includegraphics{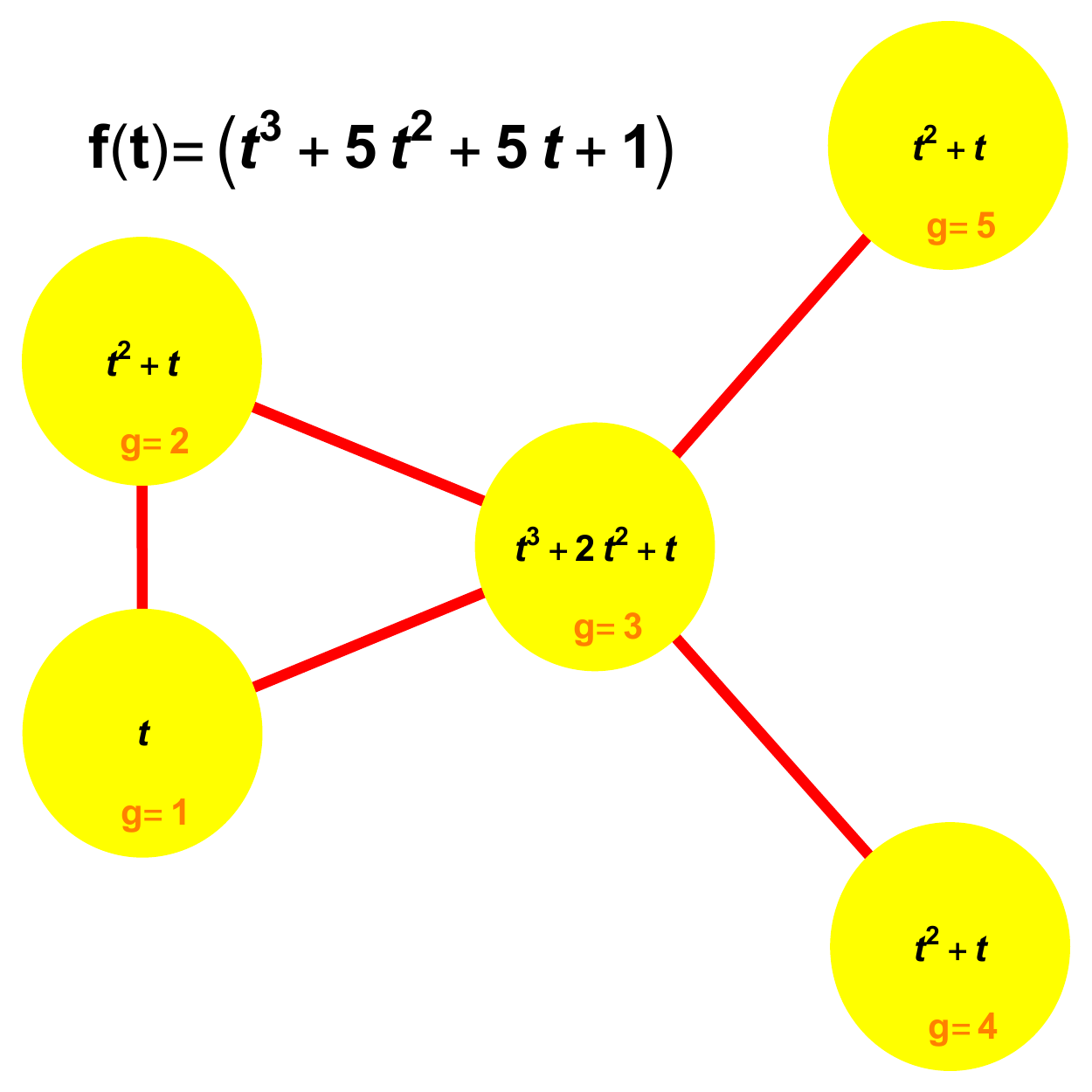}}
\label{gaussbonnetpoincarehopf}
\caption{
Illustration of the Gauss-Bonnet formula 
$f_G(t) = 1+\sum_{x \in V} F_{S(x)}(t)$
\cite{dehnsommervillegaussbonnet} and then the new Poincar\'e-Hopf formula 
$f_G(t) = 1+t \sum_{x \in V} f_{S_g(x)}(t)$ for a scalar function $g$ and
the $f$-function $f_G(t)$ of a graph $G$. In the first case, the 
curvature polynomials, in the second case the index polynomials are placed
at each vertex. 
}
\end{figure}

\paragraph{}
If $f=1+f_0 t+ \cdots + f_d t^{d+1}$ is the {\bf $f$-function} of $G=(V,E)$ 
encoding the $f$-vector $(f_0,f_1, \dots, f_d)$ of $G$ and 
$F(t)=\int_0^t f(s)ds$ is the anti-derivative of $f$, the equation 
\begin{equation}
f_G(t) = 1+\sum_{x \in V} F_{S(x)}(t) 
\label{scrumtrulescent}
\end{equation}
with unit sphere $S(x)$ is a parametrized Gauss-Bonnet formula which 
as the special case $t=-1$ evaluates to the standard Gauss-Bonnet 
formula for Euler characteristic:
$$ \chi(G) = \sum_x K(x) $$
with curvature $K(x)=\sum_{k=0}^{d} (-1)^k \frac{f_{k-1}(S(x))}{k+1}$
(understanding that $f_{-1}=1$). This curvature has appeared in \cite{Levitt1992}
but not in the context of a Gauss Bonnet. Not being aware of Levitt at first, 
it has appeared then in \cite{cherngaussbonnet}. 

\paragraph{}
A function version (\ref{scrumtrulescent}) of Gauss-Bonnet appeared in 
\cite{dehnsommervillegaussbonnet}. It was developed in order to understand the 
{\bf Dehn-Sommerville equations better}. 
The later are equivalent to the statement that 
$f(t-1/2)$ is either even or odd, a fact which quite readily 
implies that if all unit spheres $S(x)$ are Dehn-Sommerville graphs, 
then $G$ is a Dehn-Sommerville graph.  
Dehn-Sommerville graphs are ``generalized spheres" allowing to define
``generalized manifolds", graphs for which the unit spheres are Dehn-Sommerville
graphs of the same dimension. 

\section{Poincar\'e-Hopf}

\paragraph{}
Given a finite simple graph $G=(V,E)$ and a locally injective function $g$ 
from the vertex set $V$ to $\mathbb{R}$. This means that $g(x) \neq g(y)$ 
if $(x,y) \in E$.  If $f_G(t)$ is the $f$-vector of the graph and $S_g(x)$ is the graph 
generated by the part of the unit sphere $S(x)$, where $g$ is negative, the $f$-function 
$f_{S_g(x)}(t)$ of that part of the sphere is the {\bf index function} of $g$ at $x$. 
The {\bf parametrized Poincar\'e Hopf theorem} is

\begin{thm}[Parametrized Poincar\'e-Hopf]
$f_G(t) = 1+t \sum_{x} f_{S_g(x)}(t)$. 
\label{poincarehopf}
\end{thm}
\begin{proof}
The proof goes by induction with respect to the number of vertices in $G$. 
We start with $G=0$, the empty graph, where both sides of the identity are $1$. 
When adding a new vertex $x$, we increase the graph from $G$ to $G+x$.
The $f$-function changes then by $t f_{S_g(x)}(t)$ because
any $K_k$ sub-graph $H$ of $S_g(x)$ defines a $K_{k+1}$ sub-graph 
$H+x$ of $G$. 
\end{proof} 

\paragraph{}
For $t=-1$, we get the {\bf Poincar\'e-Hopf theorem} for graphs. 
The left hand side is $1-\chi(G)$. The right hand side is $1+\sum_x i_g(x)$,
where $i_g(x)=1-\chi(S_g(x))$ is the {\bf Poincar\'e-Hopf index} of $g$ at $x$. 
The Poincar\'e-Hopf theorem is 
$$ \chi(G) = \sum_x i_G(x) \; . $$

\paragraph{}
A {\bf valuation} $X$ on a graph $G$ is a map from the set of sub-graphs 
of $G$ to $\mathbb{R}$ which satisfies the valuation property
$$ X(A \cup B) + X(A \cap B) = X(A) + X(B) \; . $$  
We have in \cite{valuation} proven a Poincar\'e-Hopf formula 
$$   X(G) = \sum_{v \in V(G)} i_{X,g}(v)  \; , $$
where $i_{X,g}(x) = X(B^-_g(x))-X(S^-_g(x))$.
The {\bf unit ball} $B(v)$ of $v \in V(G)$ is defined as the sub-graph of $G$. 
It contains
$B^-(x) = S^-(x) \cup \{x\} = \{ y \in B(x) \; | \; g(y) \leq g(x) \; \}$. 
The new parametrized formula is more elegant. 

\section{Computation of the $f$-vector}

\paragraph{}
The formula in Theorem~(\ref{poincarehopf}) gives an other way, besides 
the Gauss-Bonnet Theorem~(\ref{scrumtrulescent}), to compute the $f$-vector 
of a graph recursively. The task reduces to the computation of the $f-$ function
of parts of the unit sphere. 
Given a random $g$, we can expect half of each unit sphere $S(x)$ to belong to
$S_g(x)$. In the first step, we can reduce the computation to $n$ computations 
of $f$-vectors of graphs which are of expected size $n/2$. Each of these
cases will need the computation will then involve the computation 
of an $f$-vector of graphs of size $n/4$ etc. 

\paragraph{}
A rough estimate gives a typical complexity of $O(n^{\log(n)^2})$ but this 
is not the worst case as we might be unlucky and get sphere parts $S_g(x)$ which
are large and also get unit spheres $S(x)$ which are large. 
Note however that in general the situation is much better, 
as the unit sphere $S(x)$ makes for typical 
graphs only a small local part of the network. We definitely use now this method
for our own computations of the $f$-vector as we can avoid making a list
of all the complete subgraphs of $G$ which is a task which can get us left stranded
if the graph is too large. 

\paragraph{}
Here are some computations done with the code below using {\bf Erd\"os-R\'enyi graphs}
\cite{erdoesrenyi59}, of size $n$ with edge probability $p=0.5$. 
In each case, we computed $10$ samples 
and averaged the time used to compute the $f$-vector. The code used to perform this
computation is given below. 

\begin{center}
\begin{tabular}{|c|c|} \hline
n   &  time in seconds   \\ \hline
10  &  0.065198 \\
20  &  0.382186 \\
30  &  1.743940 \\
40  &  5.80560 \\
50  &  15.4539 \\ 
60  &  30.5996 \\
70  &  53.8462 \\
80  &  121.376 \\ 
90  &  204.188 \\ 
100 &  336.029 \\ \hline
\end{tabular}
\end{center} 

\section{Index expectation}

\paragraph{}
If $(\Omega,\mu)$ is a probability space of functions $g$, denote by ${\rm E}[X]$ the
expectation of a random variable $X$. For every vertex $x$, the map $g \to i_g(x)$ 
is an example of a random variable on $(\Omega,\mu)$. 
Its expectation $K(x) = {\rm E}[i_g(x)]$ is the {\bf curvature} at $x$. 
From Poincar\'e-Hopf, we get $\chi(G) = \sum_x K(x)$. Of course, the curvature $K(x)$ depends
on the probability space. Let us call a probability space {\bf homogeneous} if it is either given by 
the product measure $(\Omega,\mu)=([0,1]^V,dx^{\rm |V|})$ or the counting measure 
on the set of all {\bf $c$-coloring} if the chromatic number of the graph is smaller 
or equal than $c$. 

\begin{lemma}
For a homogeneous probability space, ${\rm E}[ t f_{S_g(x)}(t) ] = F_{S(x)}(t)$. 
\end{lemma}
\begin{proof}
Look at each component $f_k(S(x)) t^k$. Integrating gives $f_k(S(x)) t^{k+1}/(k+1)$.
But due to the homogeneity assumption, $f_k(S_g(x)) = f_k(S(x))/(k+1)$. 
\end{proof}

\paragraph{}
This gives Gauss-Bonnet from Poincar\'e-Hopf:

\begin{coro}
$f_G(t) = 1+\sum_{x} F_{S(x)}(t)$.
\end{coro}

\paragraph{}
The link between Poincar\'e-Hopf and Gauss-Bonnet makes curvature
more intuitive as ``curvature is index expectation". A physicist 
could see the indices as integer spin values of ``particles" and curvature
as an expectation of such values when averaging over random functions $g$
which can be ``wave function probability amplitudes". 

\section{The manifold case}

\paragraph{}
In the manifold case, the discrete Poincar\'e-Hopf theorem leads to the
classical Poincar\'e-Hopf theorem from differential topology \cite{poincare85,hopf26}
(see e.g. \cite{Spivak1999}). 
In the Morse case, where the center manifold $B_g(x)=\{y \in S(x), g(y)=g(x) \}$ 
is either a $(d-2)$-sphere or a product of two spheres, the index is $\pm 1$.
When taking expectations over random functions $g$, Poincar\'e-Hopf
which again for $t=-1$ is the classical Gauss-Bonnet theorem, 
which in the manifold case leads to Gauss-Bonnet-Chern result in the continuum.

\paragraph{}
If $G$ is a discrete $2n$-manifold then for every $g$, we have the 
$(2n-2)$-dimensional {\bf center manifold} 
$B_g(x) = B(x) =\{ g(y)=g(x) \} \subset S(x)$ defined in 
\cite{KnillSard,indexformula} and an index 
$j(x)=1-\chi(B_g(x))/2$. Poincar\'e-Hopf combined with Gauss Bonnet gives
$$ \chi(G) = \sum_x j(x) = \sum_x \sum_{y \in S(x)} (1-K_g(y))  \; . $$
In the case of a $4$-dimensional manifold, then 
$K_g(y)$ is already a sectional curvature at a point $y$ of
a $2$-dimensional random surface
inside $S(x)$. This led to the insight that Euler
characteristic is a sort of {\bf average sectional curvature}.
As scalar curvature is an average over sectional curvatures, 
this corresponds to a {\bf Hilbert action} in the continuum. Euler characteristic
appears to be an interesting quantized (integer valued) functional from a physics 
point of view. 

\begin{figure}[!htpb]
\scalebox{0.5}{\includegraphics{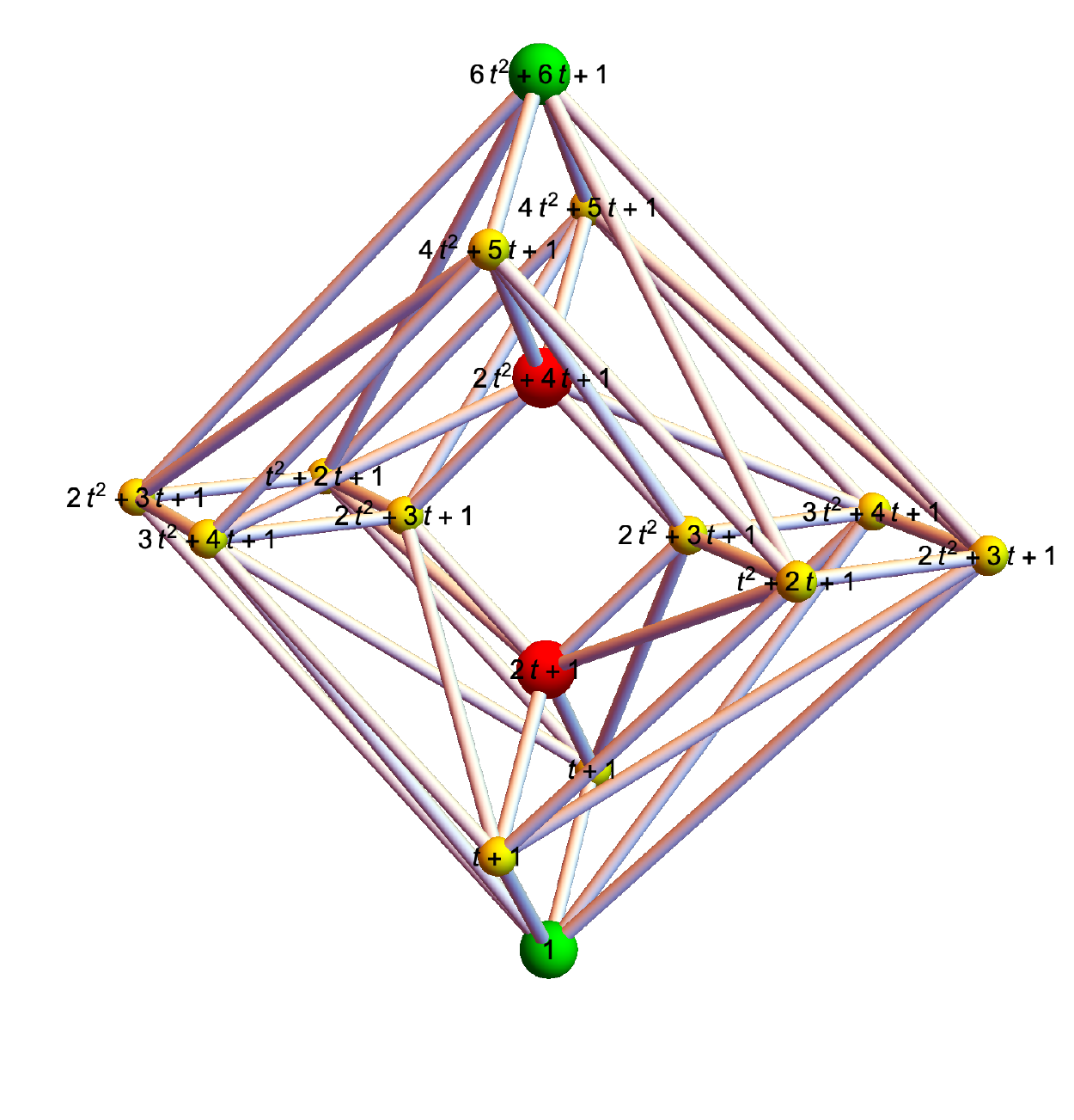}}
\scalebox{0.5}{\includegraphics{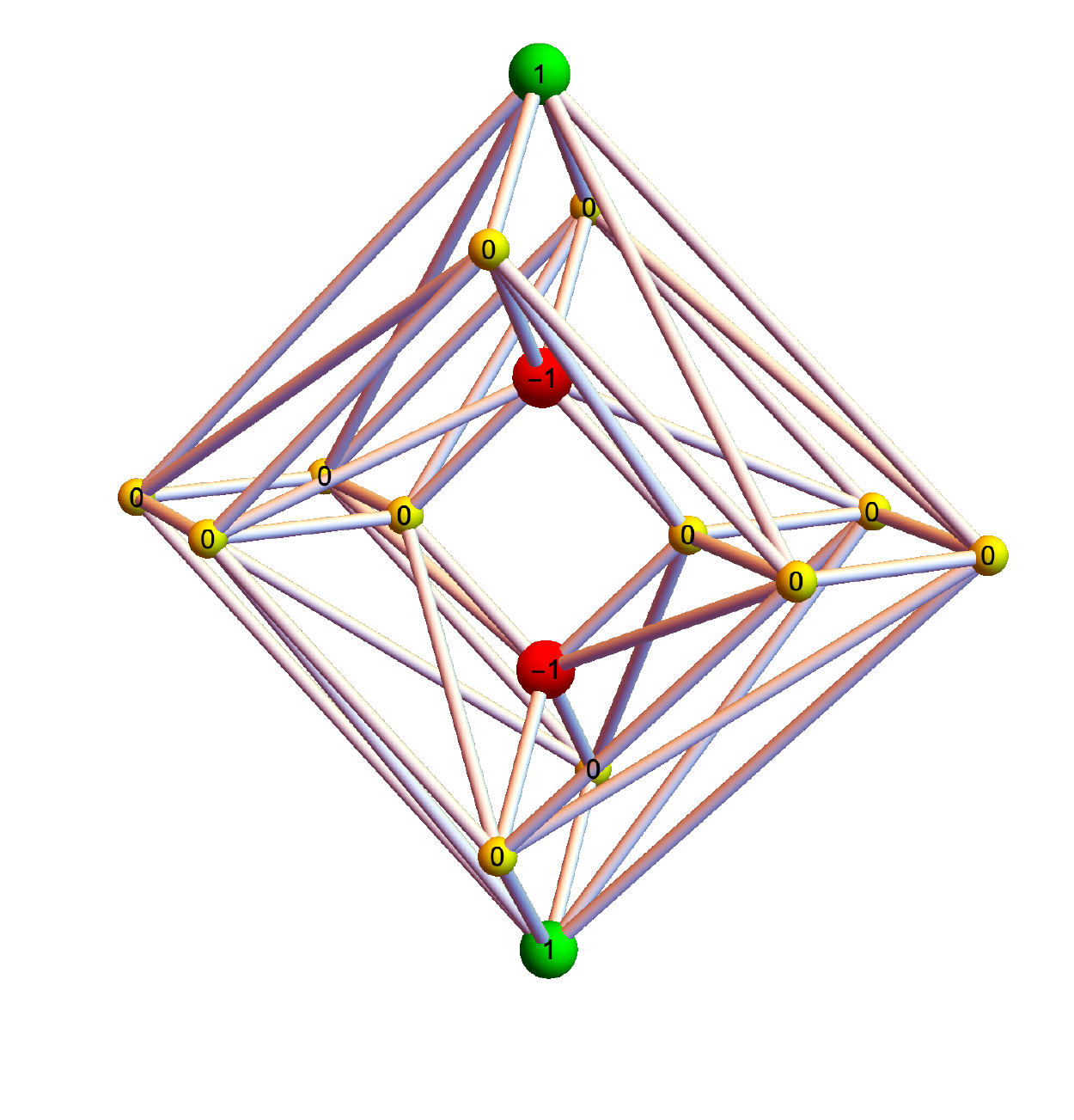}}
\label{torus}
\caption{
The Gauss-Bonnet theorem is here applied to
a regular triangulation of a smooth $2$-dimensional torus with $16$ vertices
and $f$-vector $(16,48,32)$ and $f$-function $32 t^3+48 t^2+16 t+1$. 
We took the function $g$ which is the height in an embedding so that $g$
is Morse (the center manifold $B_f(x)$ at the minimum is an empty graph
which is a $-1$ sphere; at the maximum, it is a $1$-sphere; in the hyperbolic case, 
it is the product of two $2$-spheres, a graph with 4 vertices). 
The first picture shows the index polynomials. The second one evaluates
the polynomials at $t=-1$, which gives the integer indices which in the Morse
case are always in $\{ 0,1,-1 \}$. The sum of the indices is the
Euler characteristic $\chi(G)=0$. The sum of the polynomial indices times 
$t$ plus $1$ is the $f$-function by the parametrized Poincar\'e-Hopf theorem. 
}
\end{figure}

\section{Examples}

\paragraph{}
If $G$ is a complete graph $G=K_n$, the choice of the function $g$
does not matter. We always get the indices $1,(1+t),(1+t)^2, \dots, (1+t)^{n}$
the function just determines to which vertex these polynomials are applied. 
The Poincar\'e-Hopf theorem tells that the $f$-function of $G$ is
$$ f_G(t) = (1+t)^{n+1}  = 1+ t (\sum_{k=0}^{n} (1+t)^k)  \; , $$
which can easily be checked directly by looking at the roots. 

\paragraph{}
If $G$ is a cycle graph $G=C_n$ and $g$ is a locally injective function, 
there are three type of points: minima for which $i_{x,g}(t)=1$, 
then maxima, where $i_{x,g}(t)=1+2t$ and then the regular points, 
where $i_{x,g}(t)=1+t$. Evaluated at $t=-1$, we get indices $1$,
$-1$ or $0$. We have the same number of maxima and minima so that
the sum over all $i_{x,g}(t)$ is just $n (1+t)$. Now
$$ f_G(t) = 1+n t + n t^2  = 1 + t n(1+t) $$
confirms the Poincar\'e-Hopf formula here. 


\paragraph{}
If $G$ is a graph and $G_1$ the Barycentric refinement in which the 
vertex set are the complete subgraphs of $G$ and two are connected if 
one is contained in the other, then the function $g(x) = {\rm dim}(x)$
is a locally injective function. The Poincar\'e-Hopf index is 
$\omega(x) = (-1)^{|x|} = (-1)^{{\rm dim}(x)+1}$. 
The Poincar\'e-Hopf formula 
$$ \chi(G) = \sum_{x \in V(G_1)} \omega(x) $$
just tells now that the Euler characteristic of the Barycentric refinement 
$G_1$ of $G$ is the same than the Euler characteristic of $G$. 

\paragraph{}
In that case, $i_{g,x}(t) = (1+t)^{{\rm dim}(x)} - t^{{\rm dim}(x)}$. 
The Poincar\'e-Hopf formula now connects
$$ f_G(t) = 1 + f_0 t + f_1 t^2 + \cdots + f_d t^{d+1} $$
with a sum $1+t \sum_x i_{g,x}(t)$ of index polynomials. 

\begin{figure}[!htpb]
\scalebox{0.4}{\includegraphics{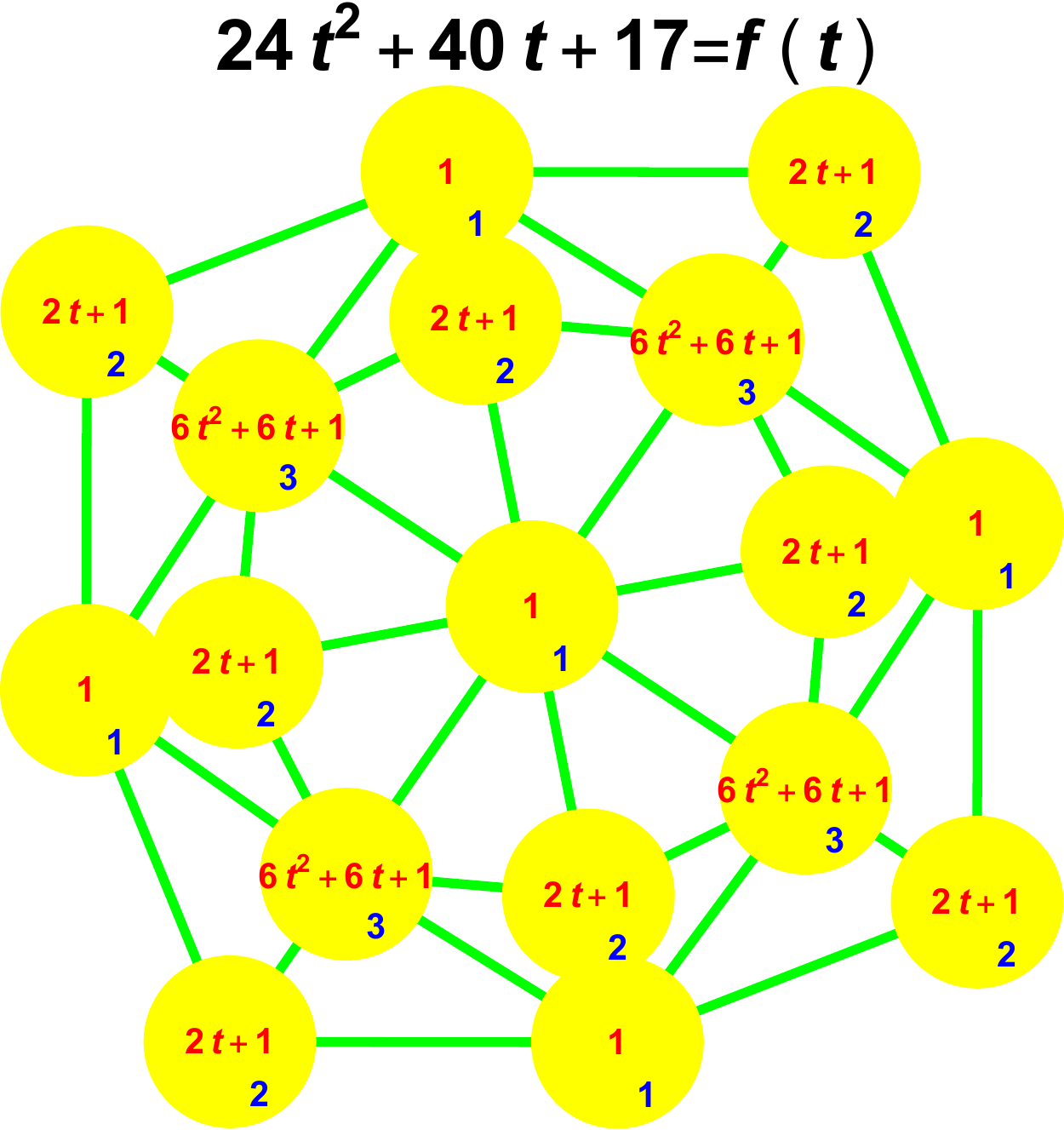}}
\scalebox{0.4}{\includegraphics{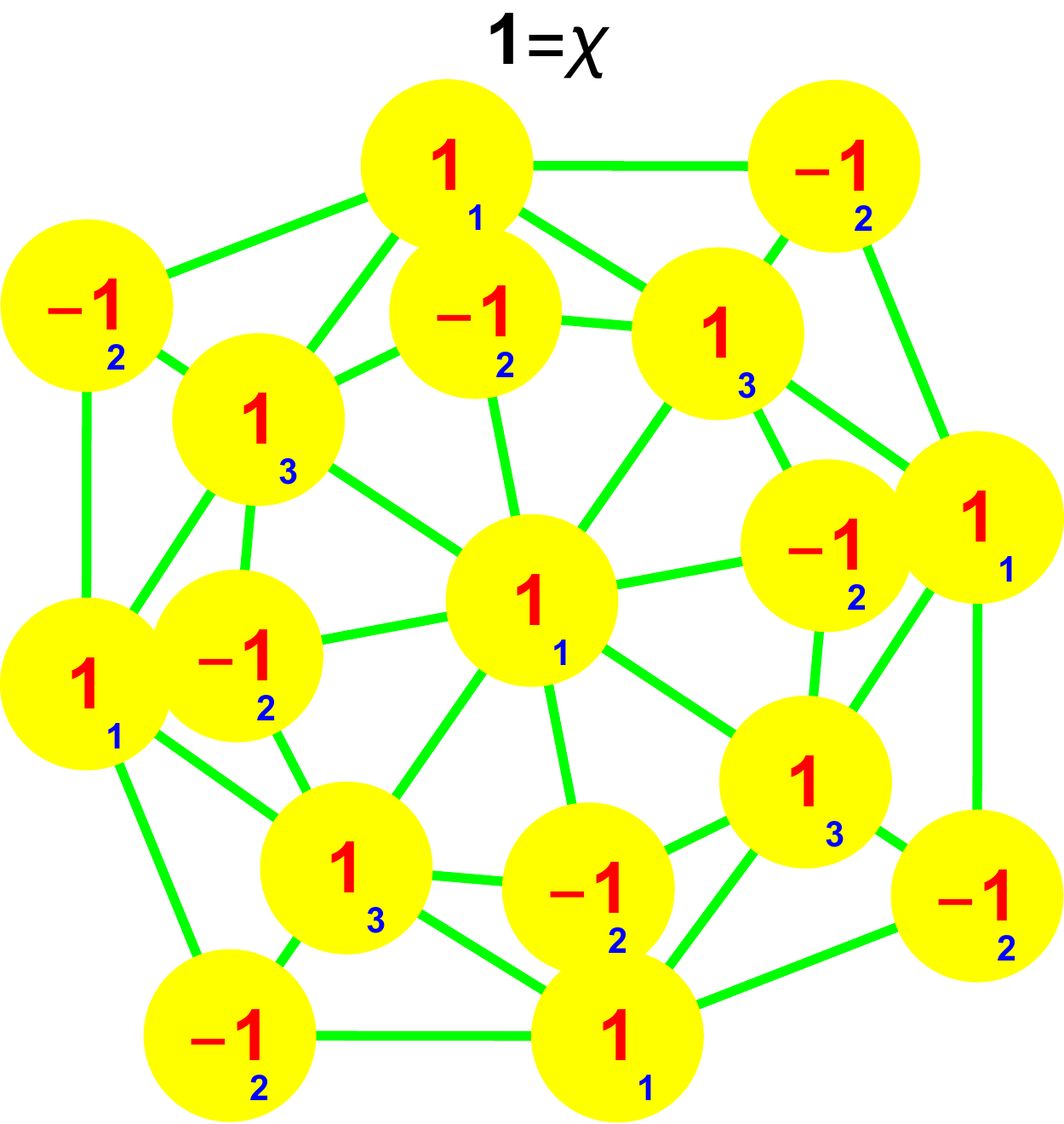}}
\label{barycentric}
\caption{
If the function $g(x)={\rm dim}(x)$ is the dimension function 
on a Barycentric refinement $G_1$ of a graph $G$,
then $i_{g,x}(t)$ at $x$ is a polynomial
of degree ${\rm dim}(x)$. We see here the case, when $G$ was a wheel graph. 
The smaller numbers inscribed below the index polynomials
are the values of $g$, which are the dimensions of the simplices
in the original graph $G$. 
}
\end{figure}

\section{Arithmetic compatibility}

\paragraph{}
The {\bf join}  $G+H==(V,E) + (W,F)$ of two finite simple graphs is the 
graph $(V \cup W, E \cup F \cup \{ (a,b), a \in V, b \in W \})$. 
Given a locally injective function $g$ on $G$ and a locally injective
function $h$ on $H$, we can look at the function $(g,h)$ which is
on $V$ given by the function $g$ and on $W$ given by the function $h$. 
There is a relation between the indices of $g$ on $G$ and $h$ on $H$ 
and $(g,h)$ on $G+H$, at least if $h$ dominates $g$. 

\paragraph{}
Of course, we have 
$f_{G+H}(t) = f_G(t) f_H(t)$ for the $f$-functions. 
But we also can have a relation between the indices: 

\begin{propo}
If $\min(h) > \max(g)$, then 
$$i_{v,g}(t) = i_{v,(g,h)}(t) $$ 
for $v \in V$ and 
$$i_{w,h}(t) = f_G(t) i_{w,(g,h)}(t)  \;  $$
for $w \in W$. 
\end{propo}
\begin{proof}
The unit spheres satisfy 
$$S_{G+H,(g,h)}(v)=S_g(G,v) $$ and
$$S_{G+H,(g,h)}(w) = G + S_h(H,w)  \; . $$
\end{proof} 

\paragraph{}
This means that if also the functions of the sum are
compatible, then the indices of the join graph are
determined from the indices of the individual components. 

\section{f-matrix computation}

\paragraph{}
The {\bf $f$-matrix} $f_{ij}$ counts the number of intersecting
$i$ and $j$-dimensional simplices in $G$. It defines the 
{\bf multi-variate $f$-function} $f_G(t,s) = \sum_{i,j} f_{ij}(G) t^i s^j$. 
The value $-f(-1,-1)= \omega(G)$ is the {\bf Wu characteristic} of $G$.
(see \cite{valuation,CohomologyWuCharacteristic}). 

\paragraph{}
There is a Gauss-Bonnet formula for Wu characteristic
$$ \omega(G) = \sum_{v \in V(G)} K(v)  \; , $$
where $K(v) = \sum_{x \sim y, v \in x} \omega(x) \omega(y)/(|x|+1)$ is the 
{\bf Wu curvature}. The sum is over all intersecting pairs $x,y$ of 
simplices in $G$ for which $x$ contains the vertex $v$. 

\paragraph{}
Because the simplex $y$ can be outside the unit sphere of $v$, the Poincar\'e-Hopf
formula for Wu characteristic is a bit more complicated.
It has been pointed out in \cite{valuation} that the most elegant formulation 
is done when the index is replaced with an index for pairs of vertices: 
$$  i_g(v,w) =  \omega(B_g(v),B_g(w)) - \omega(B_g(v),S_g(w))
              - \omega(S_g(v),B_g(w)) + \omega(S_g(v),S_f(w)) \; . $$
The Poincar\'e-Hopf formula for the Wu characteristic was then 
$$ \omega(G) = \sum_{v,w} i_g(v,w) \; . $$

\paragraph{}
To generalize this in a functorial way, 
we define the {\bf $f$-intersection function} 
$$ f_{A,B}(t,s) = \sum_{i,j} f_{ij}(A,B) t^i s^j  \; , $$
where $f_{ij}(A,B)$ is the number of pairs of $i$-simplices in $A$ and 
$j$-simplices in $B$ which do intersect. This is a quadratic form. 

\paragraph{}
Of course, $f_{G,G}(t,s) = f_G(t,s)$ is the generating function 
of the $f$-matrix $f_{ij}(G,H)$ counting the number of intersections of 
$i$-dimensional 
simplices in $G$ with $j$-dimensional simplices in $H$.  
The {\bf functional Poincar\'e-Hopf theorem} now is a result for the 
generating function of the intersection number $\omega(G,H)$:

\begin{thm}
$f_{G,H}(t,s) = \sum_{v,w} f_{B(v),B(w)}(t,s) 
                         - f_{B(v),S(w)}(t,s) 
                         - f_{S(v),B(w)}(t,s) 
                         + f_{S(v),S^-_f(w)}(t,s)$.
\end{thm}
\begin{proof}
This is a simple inclusion, exclusion principle: we can count pairs of intersecting simplices 
$(x,y)$ by looking at simplices $x$ containing a vertex $v$ and simplices $y$ 
containing a vertex $w$. The polynomial 
$$ f_{B(v),B(w)}(t,s) - f_{B(v),S(w)}(t,s) - f_{S(v),B(w)}(t,s) + f_{S(v),S^-_f(w)}(t,s) $$ 
encodes the intersection cardinalities of intersecting pairs of simplices $(x,y)$, where $x$ 
contains $v$ and $y$ contains $w$. By summing up over all vertex pairs $(v,w)$, we cover all 
intersections and get $f_{G,H}(t,s)$. 
\end{proof} 

\paragraph{}
This result is in principle useful as it allows to compute the f-matrix of a graph recursively. 
We reduce the computation to local situations, where we have intersections of unit balls
or unit spheres. While for large $n$, the computation is much faster than making a list of 
all intersecting pairs $(x,y)$ of simplices $x \in G, y \in H$, we see in practice that 
a full recursive computation like that is rather slow. 
A better algorithm is to compute $f_{G,H}(s,t)$ for smaller graphs
directly by listing all simplices and only break apart the computation into smaller parts 
for graphs of larger size. Our code example below does it that way. 

\section{Remarks}

\paragraph{}
As we have now already seen Gauss-Bonnet and Poincar\'e-Hopf, Dehn-Sommerville 
in a functorial way, one can ask about generalizing other theorems like
Euler-Poincar\'e, Riemann-Roch or Brouwer-Lefshetz \cite{brouwergraph} 
which deal with Euler characteristic.
However, for the later, cohomology is involved and the specifics of the $f$-vector 
are not important. One would then rather look at the $b$-function 
$b_G(t) = 1+b_0 t + b_1 t^2 + \cdots + b_d t^{d+1}$, where $b_k$ are the 
$k$'th Betti numbers. We have seen that the heat flow combined with the 
McKean-Singer symmetry \cite{knillmckeansinger} morphs the $f$-function 
to the $b$-function. But the equivalence of $f_G(t)$ and $b_G(t)$ only 
holds for $t=-1$ as this is based on super trace identities ${\rm str}(L^n)=0$. 

\paragraph{}
There is a recent theorem about Euler characteristic, the {\bf energy theorem} 
which tells that for an arbitrary finite abstract simplicial complex $G$ (like for example
the Whitney complex of a finite simple graph), the connection matrix $L$ (defined by 
$L(x,y) = 1$ if $x$ and $y$ intersect and $L(x,y)=0$ if they do not intersect), is 
unimodular so that the inverse matrix $g(x,y) = L^{-1}(x,y)$ is integer valued. The 
{\bf energy theorem} tells that the {\bf total potential energy} 
$\sum_{x,y}  g(x,y)$ is equal to the Euler characteristic $\chi(G)$
\cite{Helmholtz}. One can also 
{\bf ``hear the Euler characteristic"} \cite{HearingEulerCharacteristic}
because $\chi(G)$ is the number of positive eigenvalues of $L$ minus 
the number of negative eigenvalues of $L$. 

\paragraph{}
In a future article we will show that we can parametrize
these results to matrices $L$ in which $L(x,y)$ are parametrized by a parameter $t$. 
This can be done in various ways but we still debate which is the most elegant one. 
It currently appears possible that both $L$ and $g$ to have polynomial entries.

\section{Code}

\paragraph{}
The following Mathematica code illustrates the theorem. It computes
the $f$-vector recursively using Poincar\'e-Hopf. The code can be grabbed
from the ArXiv version of the paper. 

\begin{small}
\lstset{language=Mathematica} \lstset{frameround=fttt}
\begin{lstlisting}[frame=single]
UnitSphere[s_,a_]:=Module[{b},b=NeighborhoodGraph[s,a];
  If[Length[VertexList[b]]<2,Graph[{}],VertexDelete[b,a]]];
VertexFunction[s_]:=Table[2*Random[]-1,{Length[VertexList[s]]}];
ErdoesRenyi[M_,p_]:=Module[{q={},e,a,V=Range[M]},
  e=EdgeRules[CompleteGraph[M]];
  Do[If[Random[]<p,q=Append[q,e[[j]]]],{j,Length[e]}]; 
  UndirectedGraph[Graph[V,q]]];
index[g_,s_,a_,t_]:=Module[{v=VertexList[s],u,V,S={}},
   u=UnitSphere[s,a]; V=VertexList[u]; P=Position;
   Do[If[(g[[P[v,V[[k]]][[1,1]]]]-g[[P[v,a][[1,1]]]])<0,
      S = Append[S, V[[k]]]],{k,Length[V]}];
   If[Length[S]==0,1,FFunction[Subgraph[s,S],t]]];
indices[g_,s_,t_]:=Module[{v=VertexList[s], n}, n=Length[v];
   If[n == 0,{},Table[index[g,s,v[[k]],t],{k,n}]]];
FFunction[s_,t_]:=Simplify[1
  +t*Total[indices[VertexFunction[s],s,t]]];              
s0 = ErdoesRenyi[40, 0.5]; A=Timing[FFunction[s0,t]]
\end{lstlisting}
\end{small}

\paragraph{}
In the following code compute the f-matrix and multivariabe $f$-functions
also using Poincar\'e-Hopf. We also give the code to compute 
the intersection generating function $f_{G,H}(t,s)$ 
directly. It satisfies $f_{G,H}(-1,-1)=\omega(G,H)$ 
leading to the {\bf Wu characteristic} $\omega(G) = \omega(G,G)$. 

\paragraph{}
In the given version, we use the standard Wu characteristic
to compute the indices. Using the recursive version 
is slower. We see already in the demo that the new version 
is slower than just computing the functions $f_{G,H}(t,s)$
directly. At the moment, we still prefer to compute the Wu-intersection
numbers directly. 

\begin{small}
\lstset{language=Mathematica} \lstset{frameround=fttt}
\begin{lstlisting}[frame=single]
CliqueNumber[s_]:=Length[First[FindClique[s]]];
ListCliques[s_,k_]:=Module[{n,t,m,u,r,V,W,U,l={},L},L=Length;
  VL=VertexList;EL=EdgeList;V=VL[s];W=EL[s]; m=L[W]; n=L[V];
  r=Subsets[V,{k,k}];U=Table[{W[[j,1]],W[[j,2]]},{j,L[W]}];
  If[k==1,l=V,If[k==2,l=U,Do[t=Subgraph[s,r[[j]]];
  If[L[EL[t]]==k(k-1)/2,l=Append[l,VL[t]]],{j,L[r]}]]];l];
Whitney[s_]:=Module[{F,a,u,v,d,V,LC,L=Length},V=VertexList[s];
  d=If[L[V]==0,-1,CliqueNumber[s]];LC=ListCliques;
  If[d>=0,a[x_]:=Table[{x[[k]]},{k,L[x]}];
  F[t_,l_]:=If[l==1,a[LC[t,1]],If[l==0,{},LC[t,l]]];
  u=Delete[Union[Table[F[s,l],{l,0,d}]],1]; v={};
  Do[Do[v=Append[v,u[[m,l]]],{l,L[u[[m]]]}],{m,L[u]}],v={}];v];
OldWu[s1_,s2_]:=Module[{c1=Whitney[s1],c2=Whitney[s2],v=0},
  Do[Do[If[Length[Intersection[c1[[k]],c2[[l]]]]>0,
   v+=T^Length[c1[[k]]]*S^Length[c2[[l]]]],
   {k,Length[c1]}],{l,Length[c2]}];v];
\end{lstlisting}
\end{small}

\pagebreak

\paragraph{}
Here is sample code for the new part. 

\begin{small}
\lstset{language=Mathematica} \lstset{frameround=fttt}
\begin{lstlisting}[frame=single]
CompleteQ[s_]:=Binomial[Length[VertexList[s]],2]
   ==Length[EdgeList[s]];
UnitSphere[s_,a_]:=Module[{b},b=NeighborhoodGraph[s,a];
  If[Length[VertexList[b]]<2,Graph[{}],VertexDelete[b,a]]];
ErdoesRenyi[M_,p_]:=Module[{q={},e,a,V=Range[M]},
  e=EdgeRules[CompleteGraph[M]];
  Do[If[Random[]<p,q=Append[q,e[[j]]]],{j,Length[e]}];
  UndirectedGraph[Graph[V,q]]];
VertexFunction[s_]:=Table[2*Random[]-1,
  {Length[VertexList[s]]}];
WuIndex[f1_,f2_,s1_,s2_,a_, b_]:=Module[
  {vl,sp,sq,v,w,sa={},sb={},ba,bb,
  P=Position,Sg=Subgraph,A=Append,L=Length,V=VertexList},
  p[t_, u_] := P[t, u][[1,1]]; vl1=V[s1]; vl2=V[s2]; 
  sp=UnitSphere[s1,a];v=V[sp];sq=UnitSphere[s2,b];w=V[sq];
  Do[If[f1[[p[vl1,v[[k]]]]]<f1[[p[vl1,a]]],
    sa=A[sa,v[[k]]]],{k,L[v]}];
  Do[If[f2[[p[vl2,w[[k]]]]]<f2[[p[vl2,b]]],
    sb=A[sb,w[[k]]]],{k,L[w]}];
  ba = A[sa, a]; bb = A[sb, b];
  OldWu[Sg[s1,ba],Sg[s2,bb]]-OldWu[Sg[s1,sa],Sg[s2,bb]]-
  OldWu[Sg[s1,ba],Sg[s2,sb]]+OldWu[Sg[s1,sa],Sg[s2,sb]]];
WuIndices[f1_,f2_, s1_,s2_]:=Module[
  {V1=VertexList[s1],V2=VertexList[s2],L=Length}, 
  Table[WuIndex[f1,f2,s1,s2,V1[[k]],V2[[l]]],
  {k,L[V1]},{l,L[V2]}]];
Wu[s1_,s2_]:=Module[{f1=VertexFunction[s1],f2=VertexFunction[s2]},
   Total[Flatten[WuIndices[f1,f2,s1,s2]]]];
s=ErdoesRenyi[12,0.5]; 
Print[Timing[Wu[s,s]]];
Print[Timing[OldWu[s,s]]]; 
\end{lstlisting}
\end{small}

\bibliographystyle{plain}

\end{document}